\documentclass[10pt,twoside]{article}

\usepackage{amsmath,amssymb,amsfonts,amsthm}
\usepackage{a4wide}
\usepackage{graphicx}
\usepackage[all,knot,poly]{xy}

\newcommand{\N}{\mathbb{N}}
\newcommand{\R}{\mathbb {R}}
\newcommand{\dd}{D_{n,q}}
\newcommand{\ddd}{\,d_{n,q}}
\newcommand{\fft}{qt^n}
\newcommand{\ffkt}{q^{[k]_n }t^{n^k}}
\newcommand{\ffka}{q^{[k]_n }a^{n^k}}
\newcommand{\ffks}{q^{[k]_n }s^{n^k}}
\newcommand{\ffkb}{q^{[k]_n }b^{n^k}}
\newcommand{\ffla}{q^{[l]_n }a^{n^l}}
\newcommand{\fflb}{q^{[l]_n }b^{n^l}}
\newcommand{\w}{\theta}

\newcommand{\h}{n,q}

\newcommand{\To}{{\longrightarrow}}

\makeatletter
\evensidemargin \oddsidemargin
\makeatother

\newtheorem{thm}{Theorem}
\newtheorem{lem}[thm]{Lemma}

\theoremstyle{definition}
\newtheorem{rem}[thm]{Remark}
\newtheorem{exm}[thm]{Example}
\newtheorem{defn}[thm]{Definition}

\begin{document}

\thispagestyle{empty}
\setcounter{page}{1}

\noindent
{\footnotesize {\rm Accepted for publication in
Dynamics of Continuous, Discrete and Impulsive Systems,\\[-1.00mm] 
Series B (DCDIS-B), Special issue on Optimal Control and Its Applications,\\[-1.00mm] 
dedicated to the memory of Arie Leizarowitz.\\[-1.00mm] 
Submitted 04-Jan-2011; revised 30-Jun-2011; accepted 01-Jul-2011.}}

% ----------------------------------

\bigskip

\begin{center}
{\large\bf \uppercase{The power quantum calculus\\
and variational problems}}

\vskip.20in

Khaled A. Aldwoah$^{1}$,
Agnieszka B. Malinowska$^{2}$,
\ and\ Delfim F. M. Torres$^{3}$ \\[2mm]
{\footnotesize
$^{1}$Department of Mathematics, College of Science\\
Jazan University, Jazan, Saudi Arabia\\[5pt]
$^{2}$Faculty of Computer Science,
Bia{\l}ystok  University of Technology\\
15-351 Bia{\l}ystok, Poland\\[5pt]
$^{3}$Department of Mathematics,
University of Aveiro\\
3810-193 Aveiro, Portugal\\
Corresponding author email: delfim@ua.pt}
\end{center}

% ----------------------------------

{\footnotesize
\noindent
{\bf Abstract.}
We introduce the power difference calculus based on the operator
$D_{n,q} f(t) = \frac{f(qt^n)-f(t)}{qt^n -t}$,
where $n$ is an odd positive integer and $0<q<1$.
Properties of the new operator and its inverse
--- the $d_{n,q}$ integral --- are proved.
As an application, we consider power quantum Lagrangian systems
and corresponding $n,q$-Euler--Lagrange equations.\\[3pt]
}

\vskip.2in

% ----------------------------------

\section{Introduction}

Quantum derivatives and integrals play a leading role in the understanding of complex
physical systems. In 1992 Nottale introduced the theory of scale-relativity
without the hypothesis of space-time differentiability
\cite{Nottale:1992,Nottale:1999}. A rigorous mathematical
foundation to Nottale's scale-relativity theory is nowadays given
by means of a quantum calculus \cite{Ric:Del,MyID:162,Jacky:Gastao:Delfim,Kac}.
Roughly speaking, a quantum calculus substitute the classical derivative
by a difference operator, which allows to deal with sets of non differentiable curves.
For the motivation to study a non-differentiable quantum
calculus we refer the reader to \cite{Ric:Del,Jacky:Gastao:Delfim,Kac,Nottale:1992}.

Quantum calculus has several different dialects \cite{withMiguel01,Ernst,Kac}.
The most common tongue of quantum calculus
is based on the $q$-operator ($q$ stands for quantum),
which is based on the Jackson $q$-difference operator
and the associated Jackson $q$-integral \cite{Jack1,Jack2,Kac}.
The Jackson $q$-difference operator is defined by
\begin{equation*}
D_{q} f(t) = \frac{f(qt)-f(t)}{t(q-1)},\quad  \; t\neq 0 \, ,
\end{equation*}
where $q$ is a fixed number, normally taken to lie in $(0,1)$. Here
$f$ is supposed to be defined on a $q$-geometric set $A$, \textrm{i.e.},
$ A$ is a subset of $\R$ (or $\mathbb{C}$) for which $qt \in A$ whenever $t \in
A$. The derivative at zero is normally defined to be $f^\prime(0)$,
provided that $f^\prime(0) $ exists \cite{AAIM1,AAR,Ch1,Ch2,MIS,Jack1}.
Jackson also introduced the $q$-integral
\begin{equation}
\label{Jackson integral at 0}
\int_{0}^{a} f(t) d_{q} t
= a(1-q) \sum\limits_{k=0}^{\infty} q^{k} f (aq^{k})
\end{equation}
provided that the series converges \cite{Alsalam,Jack2,Kac}.
He then defined
\begin{equation}
\label{Jackson integral a b}
\int_{a}^{b} f(t) d_{q} t = \int_{0}^{b} f(t) d_{q} t -
\int_{0}^{a} f(t) d_{q} t.
\end{equation}
There is no unique canonical choice for the $q$-integral from $0$ to
$\infty$. Following Jackson we will put
\begin{equation*}
\int_{0}^{\infty} f(t)d_q t= (1-q)\sum\limits_{k= -\infty}^{\infty} q^k f(q^k),
\end{equation*}
provided the sum converges absolutely \cite{GasperRahman,Koornwinder}.
The other natural choices are then expressed by
\begin{equation*}
\int_{0}^{s\cdot\infty} f(t) d_q t= s(1-q) \sum_{k=-\infty}^{\infty} q^k
f(sq^k),\; \; s>0
\end{equation*}
(see \cite{Koornwinder}). The bilateral $q$-integral is defined by
\begin{equation*}
\int_{-s\cdot\infty}^{s\cdot\infty} f(t) d_q t= s(1-q)
\sum_{k=-\infty}^{\infty} \left[q^k f(sq^k)+q^k f(-sq^k)\right],\; \;
s>0.
\end{equation*}

The main goal of this work is to generalize the important
$q$-calculus in order to include also
the quantum calculus that results from
the $n$-power difference operator
\begin{equation}
\label{n power operator}
D_{n} f(t) =
\begin{cases}
\frac{f(t^n)-f(t)}{t^n -t} &\mbox{ if } t\in\R\setminus \{-1,0,1\},\\
f^{\prime}(t) &\mbox{ if } t\in \{-1,0,1\},
\end{cases}
\end{equation}
where $n$ is a fixed odd positive integer \cite{aldwoah}.
For that we develop a calculus
based on the new and more general proposed operator $\dd$
(see Definition~\ref{def:qp:diff}).
The class of quantum systems thus obtained has two parameters
and is wider than the standard class of quantum dynamical systems
studied in the literature.
We claim that the $\h$-calculus here introduced
offers a better mathematical modeling technique
to deal with quantum physical systems
of time-varying graininess.
We trust that our $\h$ quantum calculus will become a useful tool
to investigate more about non-conservative dynamical systems in physics
\cite{Bartos,MMAS:El:Del,JMP:El:Del,JMAA:Gastao}.

The paper is organized as follows. Our results are given in
Section~\ref{sec:MR} and Section~\ref{sec:appl}:
in \S\ref{subsec:diff} we introduce the notion
of power quantum differentiation and prove its main properties;
in \S\ref{subsec:int} we develop the notion of power quantum
integration as the inverse operation of power quantum
differentiation; and in \S\ref{ssec:EL} we obtain the
Euler--Lagrange equation for functionals defined by $\h$-derivatives
and integrals, generalizing the Euler--Lagrange equations presented
in \cite{Bang04,Bang05}. We also provide (see \S\ref{Leitmann})
an additional tool for solving power quantum variational problems, by showing that
the direct method introduced by Leitmann in the sixties of the XX
century \cite{Leitmann67}, and recently extended to different contexts
\cite{MyID:154,Leitmann01,MR2035262,abmalina:delfim,torres}, remains effective here.

% ----------------------------------

\section{The power quantum calculus}
\label{sec:MR}

For a fixed $0<q<1$, $k\in \N_0:=\N\cup\{0\}$, and a fixed
odd positive integer $n$, let us denote
\[
\w:=
\left\{\begin{array}{lcl}
\infty &\mbox{if}&n=1, \\
q^{\frac{1}{1-n}} &\mbox{if}&n\in 2\mathbb N+1,
\end{array}\right.
\quad
S:=\left\{\begin{array}{lcl}
\{0\} &\mbox{if}&n=1, \\
\{-\w,0,\w\} &\mbox{if}&n\in 2\mathbb N+1,
\end{array}\right.\]
\[\mbox{and}\quad[k]_n:=\left\{\begin{array}{lcl}
 \sum_{i=0}^{k-1}n^i&\mbox{if}&k\in \N, \\
0 &\mbox{if}&k=0.
\end{array}\right.
\]

\begin{lem}
\label{Hn To w0}
Let $h: \R \To \R$ be the function defined by $h(t):=qt^n$. Then,
$h$ is one-to-one, onto, and
$\displaystyle h^{-1} (t) =\sqrt[n]{\frac{t}{q}}$.
Moreover,
$$
h^k (t):= \underbrace{h\circ h\circ  \cdots \circ
h}_{k-\mbox{times}}(t)= q^{[k]_n} t^{n^k}
$$
and
$$
h^{-k}
(t):= \underbrace{h^{-1}\circ h^{-1}\circ  \cdots \circ
h^{-1}}_{k-\mbox{times}}(t)= q^{-n^{-k_{[k]_n}}}t^{n^{-k}}
$$
with
\[
\lim_{k\To \infty} h^k (t)
= \left\{\begin{array}{ccc}
\infty &\mbox{if}& t>\w  \\
0  &\mbox{if} & -\w<t<\w\\
-\infty &\mbox{if} & t<-\w\\
t &\mbox{if}&t\in S
\end{array}\right.
\]
and
\[
\lim_{k\To \infty} h^{-k}(t)
= \left\{\begin{array}{ccl}
\w &\mbox{if}&\;\; 0<t  \\
-\w &\mbox{if}& \;\;t<0  \\
t  &\mbox{if} & \;\;t\in S\, .
 \end{array}\right.
\]
\end{lem}

In Figure~\ref{figure 1} we illustrate the behaviour of $h^k(t)$
of Lemma~\ref{Hn To w0} in the case $-\w<t<\w$.

% ----------------------------------

\begin{figure}
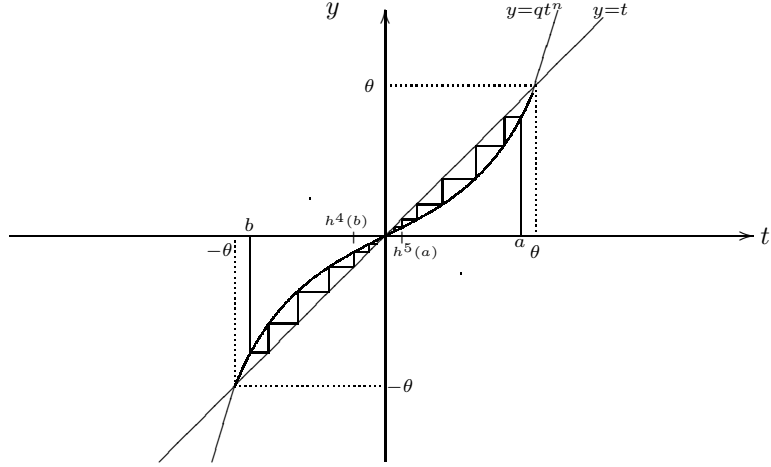

\[
\xy
{\ar@{->} (50,0)*{}; (50,60)*{}};% y-axis
{\ar@{-} (43,60)*{}; (43,60)*{y}}; {\ar@{->} (0,30)*{};
(100,30)*{\;t}}; {\ar@{-} (20,0)*{}; (80,60)*{_{y=t}}}; {\ar@{-}
(30,10)*{}; (27,0)*{}}; {\ar@{.} (30,10)*{}; (30,30)*{}}; {\ar@{.}
(70,50)*{}; (70,30)*{}}; {\ar@{.} (30,10)*{}; (50,10)*{}}; {\ar@{.}
(50,50)*{}; (70,50)*{}}; {\ar@{.} (48,50)*{}; (48,50)*{_{\theta}}};
{\ar@{.} (28,28)*{}; (28,28)*{_{-\theta}}}; {\ar@{.} (70,28)*{};
(70,28)*{_{\theta}}}; {\ar@{.} (52,10)*{}; (52,10)*{_{-\theta}}};
{\ar@{-} (69.54,48.9)*{}; (73,60)*{}}; {\ar@{-} (70,60)*{};
(70,60)*{_{y=qt^n}}};
% iterations a
{\ar@{-} (68,29)*{}; (68,29)*{_a}}; {\ar@{-} (68,30)*{};
(68,45.8)*{}}; {\ar@{-} (68,45.8)*{}; (65.8,45.8)*{}}; {\ar@{-}
(65.8,45.8)*{}; (65.8,42)*{}}; {\ar@{-} (65.8,42)*{}; (62,42)*{}};
{\ar@{-} (62,42)*{}; (62,37.6)*{}}; {\ar@{-} (62,37.6)*{};
(57.6,37.6)*{}}; {\ar@{-} (57.6,37.6)*{}; (57.6,34.2)*{}}; {\ar@{-}
(57.6,34.2)*{}; (54.2,34.2)*{}}; {\ar@{-} (54.2,34.2)*{};
(54.2,32.2)*{}}; {\ar@{-} (54.2,32.2)*{}; (52.2,32.2)*{}}; {\ar@{-}
(54.2,30)*{}; (52.2,30)*{\shortmid}}; {\ar@{-} (54.2,28)*{};
(54.2,28)*{_{_{h^5 (a)}}}}; {\ar@{-} (45,32)*{}; (45,32)*{_{_{h^4
(b)}}}}; {\ar@{-} (45.8,30)*{}; (45.8,30)*{\shortmid}}; {\ar@{-}
(52.2,32.2)*{}; (52.2,31.2)*{}}; {\ar@{-} (52.2,31.2)*{};
(51.2,31.2)*{}}; {\ar@{-} (51.2,31.2)*{}; (51.2,30.7)*{}};
% iterations b
{\ar@{-} (32,31.5)*{}; (32,31.5)*{_b}}; {\ar@{-} (32,30)*{};
(32,14.5)*{}}; {\ar@{-} (32,14.5)*{}; (34.5,14.5)*{}}; {\ar@{-}
(34.45,14.5)*{}; (34.45,18.35)*{}}; {\ar@{-} (34.45,18.35)*{};
(38.35,18.35)*{}}; {\ar@{-} (38.35,18.35)*{}; (38.35,22.5)*{}};
{\ar@{-} (38.35,22.5)*{}; (42.5,22.5)*{}}; {\ar@{-} (42.5,22.5)*{};
(42.5,25.9)*{}}; {\ar@{-} (42.5,25.9)*{}; (45.8,25.9)*{}}; {\ar@{-}
(45.8,25.9)*{}; (45.8,27.87)*{}}; {\ar@{-} (45.8,27.87)*{};
(47.87,27.87)*{}}; {\ar@{-} (47.87,27.87)*{}; (47.87,28.87)*{}};
{\ar@{-} (47.87,28.87)*{}; (48.87,28.87)*{}};
{\ar@{-} (57.6,34.2)*{}; (54.2,34.2)*{}}; {\ar@{-} (54.2,34.2)*{};
(54.2,32.2)*{}}; {\ar@{-} (54.2,32.2)*{}; (52.2,32.2)*{}}; {\ar@{-}
(52.2,32.2)*{}; (52.2,31.2)*{}}; {\ar@{-} (52.2,31.2)*{};
(51.2,31.2)*{}}; {\ar@{-} (51.2,31.2)*{}; (51.2,30.7)*{}};
(30,10)*+{};(70,50)*+{}
**\crv~Lc{~**\dir{}
~*{} (40,35)*{}&(60,25)}*{}
\endxy
\]
\caption{The iteration of $h(t)=q t^n$,\;
$t\in \mathbb R,\;n\in 2 \mathbb N+1,\;0<q<1$.}
\label{figure 1}
\end{figure}

% ----------------------------------

\subsection{Power quantum differentiation}
\label{subsec:diff}

We introduce the $\h$-power difference operator as follows:

\begin{defn}
\label{def:qp:diff}
Assume that $f$ is a real function defined on $\R$. The $\h$-power operator is given by
\begin{equation*}
\dd f(t) := \left\{
\begin{array}{lcl}
\displaystyle\frac{f(qt^n)-f(t)}{qt^n -t} &\mbox{if}&t\in \R\setminus S,\\
f^{\prime}(t) &\mbox{if}&t\in S,
\end{array}\right.
\end{equation*}
provided $f$ is differentiable at $t\in S$. If $\dd f(t)$ exists,
we say that $f$ is $\h$-differentiable at $t$.
\end{defn}

The following lemma is a direct consequence
of Definition~\ref{def:qp:diff}.

\begin{lem}
Let $f$ be a real function and $t\in \R $.
\begin{itemize}
\item[(i)] If $f$ is $\h$-differentiable at $t$, $t\in S$, then $f $ is continuous at $t$.
\item[(ii)] If $f$ is $\h$-differentiable on an interval $I\subset(-\w,\w)$, $0\in I$, and
$$
\dd f(t)=0 \text{ for } t\in I,
$$
then $f$ is a constant function on $I$.
\item[(iii)] If $f$ is $\h$-differentiable at $t$, then
$f(\fft)=f(t) + (q t^n - t) \dd f (t)$.
\end{itemize}
\end{lem}

The next theorem gives useful formulas for the computation
of $\h$-derivatives of sums, products, and quotients
of $\h$-differentiable functions.

\begin{thm}
\label{forward sum-con-multi-divi}
Assume $f, g :\R\longrightarrow \mathbb{R}$ are $\h$-differentiable  at $t \in  \R $. Then:
\begin{itemize}
\item[(i)] The sum  $ f+g:\R\longrightarrow \mathbb{R}$ is $\h$-differentiable at $t$ and
\[
\dd (f+g) (t)= \dd f (t) +\dd g (t).
\]
\item[(ii)] For any constant $c$, $c f:\R\longrightarrow \mathbb{R}$ is
$\h$-differentiable at $t$ and
\[ \dd (c f) (t)= c \dd f (t).\]
\item[(iii)] The product $ f g:\R\longrightarrow \mathbb{R}$ is
$\h$-differentiable at $t$ and
\begin{equation*}
\begin{split}
\dd (fg) (t)
&= \dd f(t) g(t)+  f (\fft) \dd g (t)\\
&= f(t) \dd g (t) + \dd f (t) g (\fft).
\end{split}
\end{equation*}
\item[(iv)] If $ g(t)g(\fft ) \neq 0 $, then $ {f}/{g}$ is
$\h$-differentiable at $t$ and
\[
\dd \left(\frac {f}{g}\right)(t)=\frac {\dd f (t) g(t)
- f(t)\dd g (t)}{g(t)g(\fft)}.
\]
\end{itemize}
\end{thm}

\begin{proof}
The proof is done by direct calculations.
\end{proof}

Next example gives explicit formulas for the $\h$-derivative
of some simple functions.

\begin{exm}
\label{derivative}
Let $f:\R \To \R$.
\begin{itemize}
\item[(i)] If $f(t)=c$ for all $t\in \R$, where $c\in \R$ is a constant, then
$\dd f (t)= 0$.

\item[(ii)] If $f(t)=t$ for all $t\in \R$, then $\dd f(t)=1$.

\item[(iii)] If $f(t)=at-b$ for
all $t \in \R$, where $a,b$ are real constants,
then by Theorem~\ref{forward sum-con-multi-divi} we have $\dd f(t)=a$.

\item[(iv)] If $f(t)=t^2$ for all $t\in \R$, then $\dd f(t)=t+q t^{n}$.

\item[(v)] If $f(t)=\frac{1}{t}$
for all $t\in \R \setminus \{0\}$, then $\dd f(t)=-\frac{1}{q t^{n+1} }$.

\item[(vi)] If $f(t)=(t+b)^m$ for all
$t\in \R$, where $b\in \R$ is a constant and $m \in \N$,
then, by induction on $m$, we obtain that
\begin{equation*}
\dd f(t)=\sum\limits_{k=0}^{m-1}(\fft+b)^k (t+b)^{m-1-k}
\end{equation*}
for $t\ne S$.
\end{itemize}
We note that by definition of the
$\h$-difference operator, one has
$$
\dd f(t)=f^\prime(t), \quad t\in S,
$$
for all functions $f$ in (i)--(vi).
\end{exm}

\begin{defn}
Let $f:\R\To \R $. We define the second $\h$-derivative by
$\dd^2 f:=\dd(\dd f)$. More generally, we define $\dd^m f$
as follows:
\begin{gather*}
\dd^0 f = f,\\
\dd^m f = \dd \dd^{m-1} f, \quad m\in \mathbb{N}.
\end{gather*}
\end{defn}

We now obtain, under certain conditions,
the formula for the $m$th $\h$-derivative of $fg$,
$m\in \mathbb{N}$.

Let $h$ be the function defined in Lemma~\ref{Hn To w0}
and let us write $h^{\circ}f$ to denote $f\circ h$.
We will denote by $\mathcal{S}^{m}_k$ the set consisting
of all possible strings of length $m$,
containing exactly $k$ times ${h^\circ}$ and $m-k$ times $\dd$.

\begin{exm}
Let $k = 2$ and $m=4$. Then,
\begin{multline*}
\mathcal{S}^{4}_{2}
=\Bigl\{\dd \dd {h^\circ} {h^\circ},\,
\dd{h^\circ}\dd{h^\circ},\, \dd{h^\circ}{h^\circ}\dd,\,
{h^\circ}{h^\circ}\dd\dd,\\
{h^\circ}\dd{h^\circ}\dd,\, {h^\circ}\dd\dd{h^\circ} \Bigr\}\, .
\end{multline*}
\end{exm}

\begin{exm}
If $m=2$, then for $k = 0,1,2$ we have
\begin{equation*}
\mathcal{S}^{2}_{0} =\Bigl\{\dd \dd \Bigr\},\quad
\mathcal{S}^{2}_{1} =\Bigl\{\dd {h^\circ},{h^\circ}\dd \Bigr\},\quad
\mathcal{S}^{2}_{2} =\Bigl\{{h^\circ}{h^\circ} \Bigr\}.
\end{equation*}
Let $f:\R\To \R$. Then,
\begin{equation*}
\begin{split}
(\dd {h^\circ}f)(t)&=\dd (f\circ h)(t),\\
({h^\circ}\dd f)(t)&=\dd f(h(t)),
\end{split}
\end{equation*}
provided all these quantities exist.
Observe that $\dd (f\circ h)(t)\neq \dd f(h(t))$.
Indeed,
\begin{equation*}
\begin{split}
\dd f(h(t))&=\frac{f(h(h(t)))-f(h(t))}{h(h(t))-h(t)}\\
&=\frac{f(h(h(t)))-f(h(t))}{(h(t)-t)(1+\dd(h(t)-t))}
=\frac{\dd (f\circ h)(t)}{(1+\dd(h(t)-t))}.
\end{split}
\end{equation*}
\end{exm}

\begin{thm}[Leibniz formula]
\label{leibniz formula Thm}
Let $\mathcal{S}^{m}_k$ be the set consisting  of all possible strings of length $m$,
containing exactly $k$ times ${h^\circ}$ and $m-k$ times $\dd$.
If $f$ is a function for which
$L f$ exist for all $L \in \mathcal{S}^{m}_k$, and function $g$ is $m$ times $\h$-differentiable, then
for all $m \in \N$ we have:
\begin{equation}
\label{leibniz formula 1}
\dd^m (fg)(t)=\sum\limits_{k=0}^{m} \left(\sum\limits_{L \in
\mathcal{S}^{m}_k} L f\right) (t)\,\dd^k g (t)\quad \mbox{ for }\; t\in \R\setminus S,
\end{equation}
and
\begin{equation}
\label{leibniz formula 2}
\dd^m (fg)(t)=\displaystyle \sum_{k=0}^{m}
\left(\begin{array}{*{20}c}
m\\
k
\end{array}\right)
\dd^{m-k} f(t )\; \dd^k  g (t)
\quad \mbox{for}\;t \in S.
\end{equation}
\end{thm}

\begin{proof}
For $t\in S$ equality \eqref{leibniz formula 2} yields
$$
(fg)^{(m)} (t)
=\sum_{k=0}^{m} \left(\begin{array}{*{20}c}
m\\ k \end{array}\right)   f^{(m-k)}(t )\;  g^{(k)} (t),
$$
which is true (the standard Leibniz formula of classical calculus).
Assume $t\notin S$. The proof is done by induction on $m$.
If $m=1$, then by Theorem~\ref{forward sum-con-multi-divi} we have
$\dd (fg)(t)=\dd f(t) g(t) + h^\circ f(t)\dd g(t)$,
\textrm{i.e.}, \eqref{leibniz formula 1} is true for
$m=1$. We now assume that \eqref{leibniz formula 1} is true for $m=s$
and prove that it is also true for $m=s+1$.
First, we note that for $k \in\N$ and $t \notin S$
\begin{equation*}
\begin{split}
&\dd^{m+1}(fg)(t) = \dd \left[\sum \limits_{k=0}^{m}
\left(\sum \limits_{L \in
\mathcal{S}^{m}_k}L f \right)(t) \;\dd^k g(t) \right] \\
&= \sum \limits_{k=0}^{m}\left[
\dd\left(\sum\limits_{L \in \mathcal{S}^{m}_k}
L f\right)(t)\; {\dd^{k}}g (t)+{h^\circ}\left(  \sum \limits_{L \in
\mathcal{S}^{m}_k}L f\right)(t)\; \dd^{k+1}g(t)
\right]\\
&= \sum\limits_{k=0}^{m}
\left(\sum\limits_{L \in \mathcal{S}^{m}_k}{\dd L}f\right)(t)\;
{\dd^k}g(t) +\sum\limits_{k=1}^{m+1} \left(  \sum\limits_{L \in
\mathcal{S}^{m}_{k-1}}  h^\circ L f\right)(t)\;{\dd^k}g(t)\\
&= \left(\sum\limits_{L \in \mathcal{S}^{m}_m} h^\circ L f\right)(t)\;
\dd^{m+1}g(t)+\left( \sum\limits_{L \in \mathcal{S}^{m}_0}\dd L f\right)(t)\; g(t)\\
& \qquad +\sum\limits_{k=1}^{m}
\left( \sum\limits_{L \in \mathcal{S}^{m}_{k-1}}{{h^\circ}L }\, f (t)
+\sum\limits_{L \in \mathcal{S}^{m}_{k}}{\dd L} \,f\right)(t)\; {\dd^k}g(t)\\
&= \left( \sum\limits_{L \in \mathcal{S}^{m+1}_{m+1}}L f\right)(t)\;
\dd^{m+1} g(t)+ \left(\sum\limits_{L \in
\mathcal{S}^{m+1}_0} L f \right)(t)\; g(t)\\
&\qquad + \sum\limits_{k=1}^{m}
\left( \sum\limits_{L \in \mathcal{S}^{m+1}_{k}} L f \right)(t)\;{\dd^k} g(t)\\
&= \sum\limits_{k=0}^{m+1} \left(\sum\limits_{L
\in \mathcal{S}^{m+1}_k}L f\right)(t) {\dd^k} g(t).
\end{split}
\end{equation*}
We conclude that \eqref{leibniz formula 1} is true for $m=s+1$. Hence, by
mathematical induction, \eqref{leibniz formula 1} holds for all
$m \in \N$ and $t\in \R\setminus S$.
\end{proof}

The standard chain rule of classical calculus does not necessarily hold true
for the $\h$-quantum calculus. For example, if we assume that
$f,g:\R \To \R $ are defined by $f(t)=t^2$ and $ g(t)=qt$,
then we have
\begin{equation*}
\begin{split}
\dd (f\circ g) (t)&=\dd (qt)^2=\dd (q^2 t^2)=q^2(t+qt^n)\\
&\ne q^2 (t+q^n t^n )= \dd f (g(t))\cdot \dd g(t).
\end{split}
\end{equation*}
However, we can derive an analogous formula of the chain rule for
our power quantum calculus.

\begin{thm}[power chain rule]
Assume $g:I \To \R$ is continuous and $\h$-differentiable,
and $f:\R \To \R$ is continuously differentiable.
Then there exists a constant $c$ between $q t^n$ and $t$ with
\begin{equation}
\label{cont. chain rule}
\dd (f\circ g) (t)=f^\prime (g(c)) \dd g (t).
\end{equation}
\end{thm}

\begin{proof}
For $t\notin S$ we have
\[
\dd (f\circ g) (t)=\frac{f(g(qt^n))-f(g(t))}{qt^n-t}.
\]
We may assume that $g(qt^n)\neq g(t)$ (because if $g(qt^n)=g(t)$,
then $\dd (f\circ g ) (t)= \dd g (t)=0$ and
\eqref{cont. chain rule} holds for any $c$ between $qt^n$ and $t$).
Then,
\begin{equation}
\label{cont. chain rule-2}
\dd (f \circ g) (t) =  \frac{f(g(qt^n))-f(g(t))}{g(qt^n)-g(t)}\cdot
\frac{g(qt^n)-g(t)}{qt^n-t}.
\end{equation}
By the mean value theorem, there exists a real
number $\tau$ between $g(t)$ and $g(qt^n)$ with
\begin{equation}
\label{cont. chain rule-3}
\frac{f(g(qt^n))-f(g(t))}{g(qt^n)-g(t)}=f^\prime (\tau ).
\end{equation}
In view of the continuity of $g$, there exists $c$ in the interval
with end points $ qt^n$ and $t $ such that $g(c)=\tau$.
Thus from \eqref{cont. chain rule-2} and \eqref{cont. chain rule-3}
we obtain \eqref{cont. chain rule}. Relation \eqref{cont. chain rule}
is true at $t,\,t \in S$, by the classical chain rule.
\end{proof}

% -----------------

\subsection{Power quantum integration}
\label{subsec:int}

In this section we are interested to study the inverse operation of
$\dd f$. We call this inverse the $\h$-integral of $f$ (or the power
quantum integral). We define the interval $I$ to be $(-\w,\w).$

\begin{defn}
Let $f: I \To \R$ and $a,b \in I$. We say that $F$
is a $\h$-antiderivative of $f$ on $I$
if $\dd F(t) =f(t)$ for all $t\in I$.
\end{defn}

>From now on we assume that all series considered along the text are
convergent.
\begin{thm}
\label{Anti derivative}
Let $f: I \To \R$ and $a,b \in I$. The function
\[
F(t) = - \displaystyle\sum_{k=0}^{\infty}
q^{[k]_n}t^{n^k}\left(q^{n^k} t^{n^k (n-1)}-1\right)
f\left(q^{[k]_n}t^{n^k}\right)
\]
is a $\h$-antiderivative of $f$ on $I$, provided $f$ is continuous
at $0$.
\end{thm}

\begin{proof}
For $t\ne 0, $ we have
\[\begin{array}{lll}
\dd F(t)&=& \displaystyle\frac{F(qt^n)-F(t)}{qt^n-t}\\
&=&\displaystyle\sum_{k=0}^{\infty} \left[-
\frac{q^{[k+1]_n}t^{n^{k+1}}}{{qt^n-t}}\left(q^{n^{k+1}} t^{n^{k+1}
(n-1)}-1\right)
f\left(q^{[k+1]_n}t^{n^{k+1}}\right)\right.\\
&&+\displaystyle \left.\frac{q^{[k]_n}t^{n^k}\left(q^{n^k} t^{n^k
(n-1)}-1\right)}{{{qt^n-t}}}
f\left(q^{[k]_n}t^{n^k}\right) \right] \\
&=&f(t).
\end{array}\]
If $t=0$, then the continuity of $f$ at $0$ implies that
\begin{equation*}
\begin{split}
\dd F(0)&= \lim_{s\To
0}\displaystyle\frac{F(s)-F(0)}{s} \\
&= \lim_{s\To 0}\displaystyle
\frac{-\sum_{k=0}^{\infty} q^{[k]_n}s^{n^k}\left(q^{n^k}
s^{n^k (n-1)}-1\right) f\left(q^{[k]_n}s^{n^k}\right)}{s}\\
&=\lim_{s\To 0}
{ -\sum_{k=0}^{\infty} q^{[k]_n}s^{n^k-1}\left(q^{n^k}
s^{n^k (n-1)}-1\right) f\left(q^{[k]_n}s^{n^k}\right)}\\
&=\lim_{s\To 0}{-\sum_{k=0}^{\infty} \left(q^{[k+1]_n}s^{n^k+n^k(n-1)-1}-q^{[k]_n}
s^{n^k -1}\right) f\left(q^{[k]_n}s^{n^k}\right)}\\
&=\lim_{s\To 0}{-\sum_{k=0}^{\infty} \left(q^{[k+1]_n}s^{n^{k+1}-1}
-q^{[k]_n} s^{n^k -1}\right) f\left(q^{[k]_n}s^{n^k}\right)}\\
&=\lim_{s\To 0} f(s)\\
&=f(0).
\end{split}
\end{equation*}
This completes the proof.
\end{proof}
We then define the indefinite $\h$-integral of $f$ by
\begin{equation*}
\int_{I} f(t)\ddd t := F(t) + C,
\end{equation*}
where $C$ is an arbitrary constant.
The definite $\h$-integral of $f$ is defined as follows.

\begin{defn}
Let $f: I \To \R$ and $a,b \in I$.
We define the $\h$-integral of $f$ from $a$ to $b$ by
\begin{equation}
\label{int a to b} \int_a ^b f(t)\ddd t := \int_{0}^b f(t)\ddd t -
\int_{0}^{a} f(t) \ddd t,
\end{equation}
where
\begin{equation}
\label{int w to x} \int_{0} ^x f(t)\ddd t := - \sum_{k=0}^{\infty}
q^{[k]_n}x^{n^k}\left(q^{n^k} x^{n^k (n-1)}-1\right)
f\left(q^{[k]_n}x^{n^k}\right), \quad x\in I\, ,
\end{equation}
provided the series at the right-hand side of
\eqref{int w to x} converge at $x=a$ and $x=b$.
\end{defn}

\begin{defn}
A function $f$ is said to be $\h$-integrable
on a subinterval $J$ of $I$ if
\[
\left|\int_{a}^{b} f(t) \ddd t\right|
< \infty \quad \mbox{for all} \;\; a,b \in J.
\]
\end{defn}

\begin{rem}
The integral formulas \eqref{int a to b} and \eqref{int w to x} yield
\eqref{Jackson integral a b} and \eqref{Jackson integral at 0}
when $n=1$; and yield the corresponding integral of the operator $D_n$
defined in \eqref{n power operator} when $q \To 1$.
\end{rem}

The following properties of the $\h$-integral are direct consequences
of the definition and provide extensions of analogous properties
of the Jackson $q$-integral \cite{Jack1,Jack2,Kac}.

\begin{lem}
\label{forwad integral props}
Let $f,g :I \To \R$ be $\h$-integrable, $k \in \R$, and $a,b,c \in I$. Then,
\begin{itemize}
\item[(i)]  $\displaystyle\int_{a}^{a}f(t)\ddd  t=0$.
\item[(ii)] $\displaystyle\int_{a}^{b}kf(t) \ddd t=k
\int_{a}^{b}f(t)\ddd t$.
\item[(iii)] $\displaystyle\int_{a}^{b}f(t)\ddd
t=-\int_{b}^{a}f(t)\ddd t$.
\item[(iv)] $\displaystyle\int_{a}^{b}f(t)\ddd t=\int_{a}^{c}f(t)\ddd t
+\int_{c}^{b}f(t) \ddd t$  for  $a\le c \le b$.
\item[(v)] $\displaystyle\int_{a}^{b}\left(f(t)+g(t)\right)\ddd t
=\int_{a}^{b}f(t)\ddd t+\int_{a}^{b} g(t) \ddd t$.
\end{itemize}
\end{lem}

\begin{thm}
\label{main_calc}
Assume that $f: I \To \R$ is continuous at $0$. Then,
\[
\int_{a}^{b} D_{n,q} f(t) d_{n,q} t
=f(b)-f(a)\quad \mbox{for all} \;\;a,b\in I.
\]
\end{thm}

\begin{proof}
First, we note that $\lim_{r\To
\infty}q^{[r]_{n}}a^{n^{r}}=\lim_{r\To
\infty}q^{[r]_{n}}b^{n^{r}}=0$. By the continuity of $f$ at $0$,
$$
\lim_{r\To 0}f(r)= \lim_{k\To \infty}f(q^{[k]_{n}}a^{n^{k}})
=\lim_{k\To \infty}f(q^{[k]_{n}}b^{n^{k}})=f(0).
$$
Thus,
\begin{equation*}
\begin{split}
\int_{a}^{b} & \dd f(t) \ddd t
= - \sum_{k=0}^{\infty} q^{[k]_n}b^{n^k}\left(q^{n^k}
b^{n^k (n-1)}-1\right)
\dd f\left(q^{[k]_n}b^{n^k}\right)\\
&\qquad + \sum_{k=0}^{\infty}
q^{[k]_n}a^{n^k}\left(q^{n^k} a^{n^k (n-1)}-1\right)
\dd f\left(q^{[k]_n}a^{n^k}\right)\\
&= \sum_{k=0}^{\infty}\left[
-q^{[k]_n}b^{n^k}\left(q^{n^k} b^{n^k (n-1)}-1\right)\frac{
f\left(q^{[k+1]_{n}}b^{n^{k+1}} \right)-f\left(q^{[k]_{n}}b^{n^{k}} \right)}
{q^{[k+1]_{n}}b^{n^{k+1}} - q^{[k]_{n}}b^{n^{k}} }\right]\\
&\qquad + \sum_{k=0}^{\infty}\left[q^{[k]_n}a^{n^k}\left(q^{n^k}
a^{n^k (n-1)}-1\right)\frac{ f\left(q^{[k+1]_{n}}a^{n^{k+1}}
\right)-f\left(q^{[k]_{n}}a^{n^{k}} \right)}
{{q^{[k+1]_{n}}a^{n^{k+1}} - q^{[k]_{n}}a^{n^{k}} }}\right]\\
&= f(b)-f\left(a\right).
\end{split}
\end{equation*}
This completes the proof.
\end{proof}

\begin{lem}
\label{ineq_lem1}
Let $s \in J\subseteq [0,\theta)$ and $g$ be $\h$-integrable on $J$.
If $0\leq|f(t)|\le g(t)$ for all $t\in \left\{q^{[k]_n}s^{n^k}: k\in\mathbb{N}_0\right\}$,
then
\begin{equation}
\label{int f le int g hahn 1}
\left|\displaystyle\int_{0}^{b} f(t)  d_{\h} t\right| \le
\displaystyle\int_{0}^{b} g(t)d_{\h} t \qquad \mbox{and
}\qquad\displaystyle\left|\int_{a}^{b} f(t) d_{\h} t\right| \le
\displaystyle\int_{a}^{b} g(t)d_{\h} t
\end{equation}
for $a,b\in\left\{q^{[k]_n}s^{n^k}: k\in\mathbb{N}_0\right\}$ with $a<b$.
Consequently, if $g(t)\ge 0$ for all $t\in \left\{q^{[k]_n}s^{n^k}: k\in\mathbb{N}_0\right\}$,
then
\begin{equation}
\label{int g ge 0 hahn 2}
\displaystyle \int_{0}^{b} g(t) d_{\h} t \ge 0\qquad \mbox{and
}\qquad\displaystyle\int_{a}^{b} g(t) d_{\h} t \ge 0
\end{equation}
for all $a,b \in \left\{q^{[k]_n}s^{n^k}: k\in\mathbb{N}_0\right\}$ such that $a<b$.
\end{lem}

\begin{proof}
If $b\in \left\{q^{[k]_n}s^{n^k}: k\in\mathbb{N}_0\right\}$,
then we can write $b=q^{[k_2]_n}s^{n^{k_2}}$ for some $k_2\in \N_{0}$.
Observe that, for all $k\in \mathbb{N}_0$,
$$
q^{[k]_n} b^{n^k}=q^{[k+k_2]_n}s^{n^{k+k_2}}
\in \left\{q^{[k]_n}s^{n^k}: k\in\mathbb{N}_0\right\}.
$$
Therefore, by assumption, we have $0\leq|f(q^{[k]_n} b^{n^k})|\le g(q^{[k]_n} b^{n^k})$
and $-q^{[k]_n}b^{n^k}\left(q^{n^k} b^{n^k (n-1)}-1\right)>0$ for all $k\in \N_{0}$.
Since $g$ is $\h$-integrable on $J$, it follows that the series
$$
\sum_{k=0}^{\infty} -q^{[k]_n}b^{n^k}\left(q^{n^k} b^{n^k
(n-1)}-1\right) f\left(q^{[k]_n}b^{n^k}\right)
$$
is absolutely convergent. Therefore,
\begin{equation*}
\begin{split}
\left|\int_{0}^{b} f(t)d_{\h} t\right|&= \left|
\sum_{k=0}^{\infty} -q^{[k]_n}b^{n^k}\left(q^{n^k} b^{n^k (n-1)}-1\right)
f\left(q^{[k]_n}b^{n^k}\right)\right| \\
&\le \sum_{k=0}^{\infty}-q^{[k]_n}b^{n^k}\left(q^{n^k} b^{n^k (n-1)}-1\right)
\left|f\left(q^{[k]_n}b^{n^k}\right)\right| \\
&\le \sum_{k=0}^{\infty}
-q^{[k]_n}b^{n^k}\left(q^{n^k} b^{n^k (n-1)}-1\right) g\left(q^{[k]_n}b^{n^k}\right)\\
&= \int_{0}^{b} g(t)d_{\h} t.
\end{split}
\end{equation*}
Now, if $a,b\in \left\{q^{[k]_n}s^{n^k}: k\in\mathbb{N}_0\right\}$
and $a<b$, then we can write $a=q^{[k_1]_n}s^{n^{k_1}}$ and
$b=q^{[k_2]_n}s^{n^{k_2}}$ for some $k_1,k_2\in \N$, $k_1>k_2$.
Hence,
\begin{equation*}
\begin{split}
\Biggl|\int_{a}^{b} & f(t)d_{\h} t\Biggr| \\
&= \left|-\sum_{k=0}^{\infty}q^{[k+k_2]_n}s^{n^{k+k_2}}\left(q^{n^{k+k_2}}
s^{n^{k+k_2} (n-1)}-1\right)f\left(q^{[k+k_2]_n}s^{n^{k+k_2}}\right)\right.\\
&\quad\left.+\sum_{k=0}^{\infty}q^{[k+k_1]_n}s^{n^{k+k_1}}\left(q^{n^{k+k_1}}
s^{n^{k+k_1} (n-1)}-1\right)f\left(q^{[k+k_1]_n}s^{n^{k+k_1}}\right)\right| \\
&= \left|\sum_{k=k_2}^{\infty}q^{[k]_n}s^{n^{k}}\left(1-q^{n^{k}}
s^{n^{k} (n-1)}\right)f\left(q^{[k]_n}s^{n^{k}}\right)\right.\\
&\quad-\left.\sum_{k=k_1}^{\infty}q^{[k]_n}s^{n^{k}}\left(1-q^{n^{k}}
s^{n^{k} (n-1)}\right)f\left(q^{[k]_n}s^{n^{k}}\right)\right|\\
&\le \sum_{k=k_2}^{k_1-1}q^{[k]_n}s^{n^{k}}\left(1-q^{n^{k}}
s^{n^{k} (n-1)}\right)\left| f\left(q^{[k]_n}s^{n^{k}}\right)\right| \\
&\le \sum_{k=k_2}^{k_1-1}q^{[k]_n}s^{n^{k}}\left(1-q^{n^{k}}
s^{n^{k} (n-1)}\right) g\left(q^{[k]_n}s^{n^{k}}\right)\\
& \qquad \mp \sum_{k=k_1}^{\infty}q^{[k]_n}s^{n^{k}}\left(1-q^{n^{k}}
s^{n^{k} (n-1)}\right) g\left(q^{[k]_n}s^{n^{k}}\right)\\
&= \sum_{k=k_2}^{\infty}q^{[k]_n}s^{n^{k}}\left(1-q^{n^{k}}
s^{n^{k} (n-1)}\right) g\left(q^{[k]_n}s^{n^{k}}\right)\\
& \qquad - \sum_{k=k_1}^{\infty}q^{[k]_n}s^{n^{k}}\left(1-q^{n^{k}}
s^{n^{k} (n-1)}\right) g\left(q^{[k]_n}s^{n^{k}}\right)
\end{split}
\end{equation*}
\begin{equation*}
\begin{split}
&= - \sum_{k=0}^{\infty}q^{[k+k_2]_n}s^{n^{k+k_2}}\left(q^{n^{k+k_2}}
s^{n^{k+k_2} (n-1)}-1\right) g\left(q^{[k+k_2]_n}s^{n^{k+k_2}}\right)\\
&\quad + \sum_{k=0}^{\infty}q^{[k+k_1]_n}s^{n^{k+k_1}}\left(q^{n^{k+k_1}}
s^{n^{k+k_1} (n-1)}-1\right) g\left(q^{[k+k_1]_n}s^{n^{k+k_1}}\right)\\
&= \int_{a}^{b} g(t)d_{\h} t.
\end{split}
\end{equation*}
To show that \eqref{int g ge 0 hahn 2} is true,
we just put $f=0$ in \eqref{int f le int g hahn 1}.
\end{proof}

\begin{lem}
\label{ineq_lem2}
Let $s \in J\subseteq (-\theta,0]$ and $g$ be $\h$-integrable on $J$.
If $0\leq|f(t)|\le g(t)$ for all
$t\in \left\{q^{[k]_n}s^{n^k}: k\in\mathbb{N}_0\right\}$, then
\begin{equation*}
\left|\displaystyle\int_{b}^{0} f(t)  d_{\h} t\right| \le
\displaystyle\int_{b}^{0} g(t)d_{\h} t \qquad \mbox{and
}\qquad\displaystyle\left|\int_{a}^{b} f(t) d_{\h} t\right| \le
\displaystyle\int_{a}^{b} g(t)d_{\h} t
\end{equation*}
for $a,b\in\left\{q^{[k]_n}s^{n^k}: k\in\mathbb{N}_0\right\}$
such that $a<b$. Consequently, if $g(t)\ge 0$ for all
$t\in \left\{q^{[k]_n}s^{n^k}: k\in\mathbb{N}_0\right\}$, then
\begin{equation*}
\displaystyle \int_{b}^{0} g(t) d_{\h} t \ge 0\qquad \mbox{and
}\qquad\displaystyle  \int_{a}^{b} g(t) d_{\h} t \ge 0
\end{equation*}
for all $a,b \in \left\{q^{[k]_n}s^{n^k}:
k\in\mathbb{N}_0\right\}$ such that $a<b$.
\end{lem}

\begin{proof}
Arguing as in the proof of Lemma~\ref{ineq_lem1},
we can show that the series
$$
\sum_{k=0}^{\infty} -q^{[k]_n}b^{n^k}\left(q^{n^k} b^{n^k
(n-1)}-1\right) f\left(q^{[k]_n}b^{n^k}\right)
$$
is absolutely convergent. Therefore, we have
\begin{eqnarray*}
\left|\displaystyle \int_{b}^{0} f(t)d_{\h} t\right|
&=&  \left|\displaystyle \int_{0}^{b} f(t)d_{\h} t\right|\\
&=& \left| \displaystyle -\sum_{k=0}^{\infty}
q^{[k]_n}b^{n^k}\left(q^{n^k} b^{n^k (n-1)}-1\right)
f\left(q^{[k]_n}b^{n^k}\right)\right| \\
&\le& \displaystyle \sum_{k=0}^{\infty}\left|
-q^{[k]_n}b^{n^k}\left(q^{n^k} b^{n^k (n-1)}-1\right)\right|
\left|f\left(q^{[k]_n}b^{n^k}\right)\right| \\
&\le&   \displaystyle \sum_{k=0}^{\infty} q^{[k]_n}b^{n^k}\left(q^{n^k}
b^{n^k (n-1)}-1\right) g\left(q^{[k]_n}b^{n^k}\right)\\
&=&  -\displaystyle \int_{0}^{b} g(t)d_{\h} t
=\displaystyle \int_{b}^{0} g(t)d_{\h} t.
\end{eqnarray*}
The rest of the proof can be done similarly to the proof of
Lemma~\ref{ineq_lem1}.
\end{proof}

It should be noted that the inequality
\begin{equation*}
\left|\int_a^b f(t)d_{\h} t \right|\le \int_a^b| f(t)| d_{\h} t
\quad \mbox{for all} \;a,b\in I
\end{equation*}
is not always true. For example, fix $n=1$, $0<q<1$, and
define the function $f : [0,1]\To\R$ by
\begin{equation*}
f(x)=\left\{
\begin{array}{cl}
\frac{1}{1-q}\left(4q^{-n}x-(1+3q)\right),& q^{n+1}\le
x\le\frac{q^n(1+q)}{2},\,n\in\N,\\
\frac{4}{1-q}\left(-xq^{-n}+1\right)-1,&\frac{q^{n}(1+q)}{2}\leqslant
x\leqslant q^n,\,n\in\N,\\
0,&x=0,
\end{array}\right.
\end{equation*}
(see \cite{AAIM1}). It follows that $f$ is $\h$-integrable on $[0,1]$,
$f(q^n)=-1$, and $\displaystyle f\left(\frac{1+q}{2}q^n\right)=1$ for
all $n\in\N$. By a direct calculation one can see that
\[
\int_{\frac{1+q}{2}}^{1} f(t)\, d_{\h}t
=-\frac{3+q}{2}\quad
\text{and}\quad \int_{\frac{1+q}{2}}^{1}|f(t)|\, d_{\h}t
=\frac{1-q}{2}.
\]
Thus,
\[
\left|\int_{\frac{1+q}{2}}^{1}f(t)\,d_{\h}t\right|
>\int_{\frac{1+q}{2}}^{1}|f(t)|\,d_{\h}t.
\]

\begin{lem}
Let $f,g :I \To \R$.
\begin{itemize}
\item[(i)] If functions $f$ and $g$ are $\h$-differentiable,
then the following integration by parts formula holds:
\begin{equation}
\label{integral by part}
\int_{a}^{b}f(t)\dd g(t)\ddd t
= \left.f(t)g(t)\right|_{a}^{b}
- \int_{a}^{b}\dd f (t) g(\fft)\ddd t,\;\;\;a,b \in I.
\end{equation}
\item[(ii)] If $f$ is continuous at $0$, then for $t \in I$
\[
\int_{t}^{\fft} f(r) \ddd r = (qt^n- t)  f(t) .
\]
\end{itemize}
\end{lem}

\begin{proof}
\begin{itemize}
\item[(i)]  By Theorem~\ref{main_calc} we have
$$
\int_{a}^{b} D_{n,q} (fg)(t) d_{n,q} t
=(fg)(b)-(fg)(a).
$$
On the other hand, by (iii) of Theorem~\ref{forward sum-con-multi-divi}
and (v) of Theorem~\ref{forwad integral props},
$$
\int_{a}^{b}\dd (fg) (t)= \int_{a}^{b}f(t) \dd g (t)d_{n,q} t
+ \int_{a}^{b}\dd f (t) g (\fft)d_{n,q} t.
$$
Combining these two equalities we get the desired formula.
\item[(ii)]
\begin{equation*}
\begin{split}
\int_{t}^{\fft} & f(s) \ddd s
= \int_{0}^{\fft} f(s) \ddd s
- \displaystyle\int_{0}^{t} f(s) \ddd s\\
&= \sum_{k=0}^{\infty}
\Biggl[ q^{[k]_n}t^{n^k}\left(q^{n^k} t^{n^k (n-1)}-1\right)
f\left(q^{[k]_n}t^{n^k}\right)\\
& \;\;\;\;\;\;\;\;\qquad -  q^{[k+1]_n}t^{n^{k+1}}\left(q^{n^{k+1}}
t^{n^{k+1} (n-1)}-1\right) f\left(q^{[k+1]_n}t^{n^{k+1}}\right)\Biggr]\\
&=  (qt^n-t)f(t).
\end{split}
\end{equation*}
\end{itemize}
\end{proof}

% ----------------------

\section{The power quantum variational calculus}
\label{sec:appl}

The calculus of variations is a classical subject of mathematics
with many applications in physics, economics, biology, and engineering
\cite{Leizarowitz1985,Leizarowitz1989,Weinstock}. Although an old theory,
is very much alive and still evolving --- see, \textrm{e.g.},
\cite{MyID:182,Ric:Del:09,Leizarowitz2003,NataliaHigherOrderNabla}.

Several quantum variational problems have been recently posed and studied
\cite{Ric:Del,Bang04,Bang05,Cresson,Jacky:Gastao:Delfim}.
Here we give one application of the power quantum calculus
which we derived in Section~\ref{sec:MR}, introducing
the power quantum variational calculus and
proving a quantum analog of the Euler--Lagrange equation
(Section~\ref{ssec:EL}). This provides a necessary optimality
condition for local minimizers. Direct methods can also
be developed for our power quantum variational calculus,
allowing to obtain global minimizers
for certain classes of problems (Section~\ref{Leitmann}).

As in Section~\ref{subsec:int}, we define the interval $I$
to be $(-\w,\w).$ Let $a,b\in I$ with $a< b$. We define the $\h$-interval by
$$
[a,b]_{\h}:=\left\{\ffka: k\in\mathbb{N}_{0}\right\}\cup\left\{\ffkb:
k\in\mathbb{N}_{0}\right\}\cup\{0\}.
$$
Let $\mathcal{D}([a,b]_{\h},\mathbb{R})$
be the set of all real valued
functions continuous and bounded on $[a,b]_{\h}$.

\begin{lem}[Fundamental Lemma of the power quantum variational calculus]
\label{lemma:DR} Let $f\in \mathcal{D}([a,b]_{\h},\mathbb{R})$. One
has $\int_{a}^{b}f(t)g(\fft)\ddd t=0$ for all functions $g\in
\mathcal{D}([a,b]_{\h},\mathbb{R})$ with $g(a)=g(b)=0$ if and only
if $f(t)=0$ for all $t\in[a,b]_{\h}$.
\end{lem}

\begin{proof}
The implication ``$\Leftarrow$'' is obvious. Let us prove the
implication ``$\Rightarrow$''. Suppose, by contradiction, that
$f(c)\neq 0$ for some $c\in[a,b]_{\h}$.\\
\emph{Case I}. If $c\neq 0$, then  without loss of generality we can
assume that $c=\ffka$ for $k\in\mathbb{N}_{0}$. Define
$$g(t)= \left\{
\begin{array}{lcl}
f\left(\ffka\right)  & & \text{if } t=q^{[k+1]_n}a^{n^{k+1}}\\
0 & & \mbox{otherwise}\, .
\end{array} \right.
$$
Since $a\neq 0$, we see that
$$
\int_{a}^{b}f(t)g(\fft)\ddd t= \ffka\left(q^{n^k}a^{n^k(n-1)}
-1\right)\left(f\left(\ffka\right)\right)^2\neq 0,
$$
which is a contradiction.\\
\emph{Case II}. If $c=0$, then without loss of generality we can
assume that $f(0)>0$. We know that (see Lemma~\ref{Hn To w0})
$$\lim_{k\rightarrow\infty}\ffka=\lim_{k\rightarrow\infty}\ffkb=0.$$
As $f$ is continuous at $0$,
$$
\lim_{k\rightarrow\infty}f\left(\ffka\right)
=\lim_{k\rightarrow\infty}f\left(\ffkb\right)=f(0).
$$
Therefore, there exists $N\in\mathbb{N}$
such that for all $l>N$ the inequalities
$$
f\left(\ffla\right)>0,\quad f\left(\fflb\right)>0,
$$
hold. Let us fix $k>N$. If $a\neq 0$, then we define
$$
g(t) = \left\{
\begin{array}{lcl}
f\left(\ffka\right)  & & \text{if } t=q^{[k+1]_n}a^{n^{k+1}}\\
0 & & \mbox{otherwise}\, .
\end{array} \right.
$$
Since $a\neq 0$, we see that
$$
\int_{a}^{b}f(t)g(\fft)\ddd t
=\ffka\left(q^{n^k}a^{n^k(n-1)}-1\right)\left(f\left(\ffka\right)\right)^2
\neq 0,
$$
which is a contradiction. If $a=0$, then we define
$$
g(t)= \left\{
\begin{array}{lcl}
f\left(\ffkb\right)  & & \text{if } t=q^{[k+1]_n}b^{n^{k+1}}\\
0 & & \mbox{otherwise}\, .
\end{array} \right.
$$
Since $b\neq 0$, we obtain
\begin{equation*}
\begin{split}
\int_{a}^{b}f(t)g(\fft)\ddd t
&=\int_{0}^{b}f(t)g(\fft)\ddd t\\
&=-\ffkb\left(q^{n^k}b^{n^k(n-1)}-1\right)\left(f\left(\ffkb\right)\right)^2
\neq 0,
\end{split}
\end{equation*}
which is a contradiction.
\end{proof}

Let $\mathbb{E}([a,b]_{\h},\mathbb{R})$ be the linear space of
functions $y\in \mathcal{D}([a,b]_{\h},\mathbb{R})$ for which the
$\h$-derivative is continuous and bounded on $[a,b]_{\h}$. We equip
$\mathbb{E}([a,b]_{\h},\mathbb{R})$ with the norm
\begin{equation*}
   \|y\|_{1}=\sup_{t\in [a,b]_{\h}}|y(t)|+\sup_{t\in [a,b]_{\h}}|\dd y(t)|.
\end{equation*}

The following definition and lemma are similar
to those of \cite{MalinowskaTorres}.

\begin{defn}
Let
$g:[s]_{\h}\times(-\bar{\epsilon},\bar{\epsilon})
\rightarrow\mathbb{R}$, where
$$
[s]_{\h}:=\left\{\ffks: k\in\mathbb{N}_{0}\right\}.
$$
We say that $g(t,\cdot)$ is continuous in $\epsilon_{0}$, uniformly in
$t$, if and only if for every $\varepsilon>0$ there exists $\delta>0$ such that
$|\epsilon-\epsilon_{0}|<\delta$ implies
$|g(t,\epsilon)-g(t,\epsilon_{0})|<\varepsilon$ for all $t\in[s]_{\h}$.
Furthermore, we say that $g(t,\cdot)$ is differentiable at
$\epsilon_{0}$, uniformly in $t$, if and only if for every $\varepsilon>0$ there
exists $\delta>0$ such that $0<|\epsilon-\epsilon_{0}|<\delta$ implies
$$
\left|\frac{g(t,\epsilon)-g(t,\epsilon_{0})}{\epsilon
-\epsilon_{0}}-\partial_2g(t,\epsilon_{0})\right|<\varepsilon,
$$
where $\partial_2g=\frac{\partial g}{\partial \epsilon}$ for all
$t\in[s]_{\h}$.
\end{defn}

\begin{lem}
\label{fun} Let $s \in I$. Assume $g(t,\cdot)$ is differentiable at
$\epsilon_0$, uniformly in $t$ in $[s]_{\h}$, and that
$G(\epsilon):=\int_{0}^sg(t,\epsilon)d_{q,\omega}t$, for $\epsilon$ near
$\epsilon_0$, and $\int_{0}^s\partial_2g(t,\epsilon_0)d_{q,\omega}$
exist. Then, $G(\epsilon)$ is differentiable at $\epsilon_0$ with
$G'(\epsilon_0)=\int_{0}^s\partial_2g(t,\epsilon_0)\ddd t$.
\end{lem}

\begin{proof}
Without loss of generality we can assume that $s>0$.
Let $\varepsilon>0$ be arbitrary. Since $g(t,\cdot)$ is
differentiable at $\epsilon_{0}$, uniformly in $t$, there exists
$\delta>0$ such that, for all $t\in[s]_{\h}$ and for
$0<|\epsilon-\epsilon_{0}|<\delta$, the following inequality holds:
$$
\left|\frac{g(t,\epsilon)-g(t,\epsilon_{0})}{\epsilon-\epsilon_{0}}
-\partial_2g(t,\epsilon_{0})\right|<\frac{\varepsilon}{s}.
$$
Applying
Lemma~\ref{forwad integral props} and Lemma~\ref{ineq_lem1}, for
$0<|\epsilon-\epsilon_{0}|<\delta$, we have
\begin{equation*}
\begin{split}
&\left|\frac{G(\epsilon)-G(\epsilon_{0})}{\epsilon-\epsilon_{0}}
-G'(\epsilon_{0})\right|\\
&\quad =\left|\frac{\int_{0}^sg(t,\epsilon)\ddd t-\int_{0}^s
g(t,\epsilon_{0})\ddd t}{\epsilon-\epsilon_{0}}
-\int_{0}^s\partial_2g(t,\epsilon_{0})\ddd t\right|\\
&\quad =\left|\int_{0}^s\left[\frac{g(t,\epsilon)
-g(t,\epsilon_{0})}{\epsilon-\epsilon_{0}}
-\partial_2g(t,\epsilon_{0})\right]\ddd t\right|\\
&\quad <\int_{0}^s\frac{\varepsilon}{s}\ddd t
=\frac{\varepsilon}{s}\int_{0}^s1\ddd t =\varepsilon.
\end{split}
\end{equation*}
Hence, $G(\cdot)$ is differentiable at $\epsilon_0$ and
$G'(\epsilon_0)=\int_{0}^s\partial_2g(t,\epsilon_0)\ddd t$.
\end{proof}

% ----------------------------------------------------

\subsection{The power quantum Euler--Lagrange equation}
\label{ssec:EL}

We consider the variational problem of finding minimizers
(or maximizers) of a functional
\begin{equation}
\label{vp}
\mathcal{L}[y]
=\int_{a}^{b}f(t,y(\fft),\dd y(t))\ddd t ,
\end{equation}
over all $y\in\mathbb{E}([a,b]_{\h},\mathbb{R})$ satisfying the
boundary conditions
\begin{equation}
\label{bc}
y(a)=\alpha,\,
\quad y(b)=\beta, \quad \alpha,
\beta \in \mathbb{R},
\end{equation}
where $f:[a,b]_{\h}\times\mathbb{R}\times\mathbb{R}\rightarrow
\mathbb{R}$ is a given function. A function $y\in
\mathbb{E}([a,b]_{\h},\mathbb{R})$ is said to be admissible if it
satisfies endpoint conditions \eqref{bc}. Let us denote by
$\partial_2f$ and $\partial_3f$, respectively, the partial
derivatives of $f(\cdot,\cdot,\cdot)$ with respect to its second and
third argument. In the sequel, we assume that $(u,v)\rightarrow
f(t,u,v)$ is a $C^1(\mathbb{R}^{2}, \mathbb{R})$ function for any $t
\in [a,b]_{\h}$ and $f(\cdot,y(\cdot),D_{\h}y(\cdot))$,
$\partial_2f(\cdot,y(\cdot),D_{\h}y(\cdot))$, and
$\partial_3f(\cdot,y(\cdot),D_{\h}y(\cdot))$ are continuous and
bounded on $[a,b]_{\h}$ for all admissible functions $y(\cdot)$.
We say that $p\in \mathbb{E}([a,b]_{\h},\mathbb{R})$
is an admissible variation for \eqref{vp}--\eqref{bc}
if $p(a)=p(b)=0$.

For an admissible variation $p$, we define function $\phi : \,
(-\bar{\varepsilon},\bar{\varepsilon})\rightarrow \mathbb{R}$ by
$$
\phi(\varepsilon) = \phi(\varepsilon;y,p) :=\mathcal{L}[y +
\varepsilon p].
$$
The first variation of problem \eqref{vp}--\eqref{bc}
is defined by
$$
\delta\mathcal{L}[y,p]:=\phi(0;y,p)=\phi'(0).
$$
Observe that,
\begin{equation*}
\begin{split}
\mathcal{L}[y + \varepsilon p]
&= \int_a^b f(t,y(qt^n)+\varepsilon
p(qt^n),\dd y(t)
+ \varepsilon \dd p(t)) \ddd  t\\
&=\int_{0}^b f(t,y(qt^n)+\varepsilon p(qt^n),\dd y(t)
+ \varepsilon \dd p(t)) \ddd  t\\
&\qquad -\int_{0}^a f(t,y(qt^n)+\varepsilon p(qt^n),\dd y(t)
+ \varepsilon \dd p(t)) \ddd t.
\end{split}
\end{equation*}
Writing
\begin{equation*}
\mathcal{L}_b[y + \varepsilon p]=\int_{0}^b f(t,y(qt^n)+\varepsilon
p(qt^n),\dd y(t)+ \varepsilon \dd p(t)) \ddd  t
\end{equation*}
and
\begin{equation*}
\mathcal{L}_a[y + \varepsilon p]=\int_{0}^a f(t,y(qt^n)+\varepsilon
p(qt^n),\dd y(t)+ \varepsilon \dd p(t)) \ddd  t,
\end{equation*}
we have $$\mathcal{L}[y + \varepsilon p]=\mathcal{L}_b[y +
\varepsilon p]-\mathcal{L}_a[y + \varepsilon p].$$ Therefore,
\begin{equation}
\label{var}
\delta\mathcal{L}[y,p]
=\delta\mathcal{L}_b[y,p]-\delta\mathcal{L}_a[y, p].
\end{equation}
Knowing \eqref{var}, the following lemma is a direct
consequence of Lemma~\ref{fun}.

\begin{lem}
\label{asump} Put $g(t,\varepsilon)=f\left(t,y(qt^n)+\varepsilon
p(qt^n),\dd y(t)+ \varepsilon \dd p(t)\right)$ for $\varepsilon \in
(-\bar{\varepsilon},\bar{\varepsilon})$. Assume that:
\begin{itemize}
\item[(i)] $g(t,\cdot)$ is differentiable at $0$, uniformly in
$t \in [a]_{\h}$, and $g(t,\cdot)$ is differentiable at $0$,
uniformly in $t \in [b]_{\h}$;
\item[(ii)] $\mathcal{L}_a[y + \varepsilon p]$ and $\mathcal{L}_b[y
+ \varepsilon p]$, for $\varepsilon$ near $0$, exist;
\item[(iii)] $\int_{0}^a\partial_2g(t,0)\ddd t$ and
$\int_{0}^b\partial_2g(t,0)\ddd t$ exist.
\end{itemize}
Then,
\begin{multline*}
\delta\mathcal{L}[y,h]=\int_a^b \Biggl[
\partial_2 f(t,y(qt^n),\dd y(t)) p(qt^n) \\
+ \partial_3 f(t,y(qt^n),\dd y(t)) \dd p(t)\Biggr]\ddd t.
\end{multline*}
\end{lem}

In the sequel, we always assume, without mentioning it explicitly,
that variational problems satisfy the assumptions of
Lemma~\ref{asump}.

\begin{defn}
An admissible function $\tilde{y}$ is said to be a local minimizer
(resp. a local maximizer) to problem \eqref{vp}--\eqref{bc} if there
exists $\delta > 0$ such that $\mathcal{L}[\tilde{y}]\leq \mathcal{L}[y]$
(resp. $\mathcal{L}[\tilde{y}]\geq \mathcal{L}[y]$) for all
admissible $y$ with $\|y-\tilde{y}\|_{1}<\delta$.
\end{defn}

The following result offers a necessary condition for local
extremizer.

\begin{thm}[A necessary optimality condition for problem \eqref{vp}--\eqref{bc}]
\label{nec:con} Suppose that the optimal path to problem
\eqref{vp}--\eqref{bc} exists and is given by $\tilde{y}$. Then,
$\delta\mathcal{L}[\tilde{y},p]=0$.
\end{thm}

\begin{proof}
The proof is similar to the one found in \cite{MalinowskaTorres}.
\end{proof}

\begin{thm}[Euler--Lagrange equation for problem \eqref{vp}--\eqref{bc}]
\label{thm:mr}
Suppose that $\tilde{y}$ is an optimal path
to problem \eqref{vp}--\eqref{bc}. Then,
\begin{equation}
\label{Euler}
\dd\partial_3f(t,\tilde{y}(\fft),\dd\tilde{y}(t))
=\partial_2f(t,\tilde{y}(\fft),\dd\tilde{y}(t))
\end{equation}
for all $t \in [a,b]_{\h}$.
\end{thm}

\begin{proof}
Suppose that $\mathcal{L}$ has a local extremizer $\tilde{y}$.
Consider the value of $\mathcal{L}$ at a nearby function $y =
\tilde{y} + \varepsilon p$, where $\varepsilon\in \mathbb{R}$ is a
small parameter, $p\in \mathbb{E}$, and $p(a) =p(b)= 0$. Let
$$
\phi(\varepsilon) = \mathcal{L}[\tilde{y} + \varepsilon p] =
\int_a^b f(t,\tilde{y}(\fft)+\varepsilon p(\fft),\dd\tilde{y}(t)+
\varepsilon \dd p(t)) \ddd t.
$$
By Theorem~\ref{nec:con}, a necessary condition for $\tilde{y}$ to be
an extremizer is given by
\begin{equation}
\label{eq:FT} \left.\phi'(\varepsilon)\right|_{\varepsilon=0} = 0
\Leftrightarrow \int_a^b \left[ \partial_2f(\cdots) p(\fft) +
\partial_3f(\cdots) \dd p(t)\right]\ddd t = 0 \, ,
\end{equation}
where $(\cdots) = \left(t,\tilde{y}(\fft),\dd \tilde{y}(t)\right)$.
Integration by parts (see \eqref{integral by part}) gives
\begin{equation*}
\int_a^b \partial_3f(\cdots) \dd p(t)\ddd t
=\left.\partial_3f(\cdots)p(t)\right|_{t=a}^{t=b}-\int_a^b
 \dd \partial_3f(\cdots)p(\fft)\ddd t.
\end{equation*}
Because $p(a) = p(b)= 0$, the necessary condition \eqref{eq:FT} can
be written as
\begin{equation*}
0 = \int_a^b \left(\partial_2f(\cdots)- \dd
\partial_3f(\cdots)\right) p(\fft)\ddd t
\end{equation*}
for all $p$ such that $p(a) = p(b) = 0$. Thus, by
Lemma~\ref{lemma:DR}, we have
\begin{equation*}
\partial_2f(\cdots)-\dd \partial_3f(\cdots)=0
\end{equation*}
for all $t\in[a,b]_{\h}$.
\end{proof}

\begin{rem}
Analogously to the classical calculus of variations \cite{Weinstock},
to the solutions of the Euler--Lagrange equation \eqref{Euler}
we call \emph{(power quantum) extremals}.
\end{rem}

\begin{rem}
If the function under the sign of integration $f$ (the Lagrangian)
is given by $f = f(t,y_1,\ldots, y_m, \dd y_1,\ldots,\dd y_m)$, then
the necessary optimality condition is given by $m$ equations similar
to \eqref{Euler}, one equation for each variable.
\end{rem}

\begin{exm}
\label{example1} Let us fix $n$, $q$, such that $1\in I$. Consider
the problem
\begin{equation}\label{AD}
\text{minimize} \quad
\mathcal{L}[y]=\int_{0}^{1}\left(y(\fft)
+\frac{1}{2}(\dd y(t))^2\right)\ddd t
\end{equation}
subject to the boundary conditions
\begin{equation}
\label{eq:bc}
    y(0)=0,  \quad y(1)=\beta,
\end{equation}
where $\beta\in \R$. If $y$ is a local minimizer to problem
\eqref{AD}--\eqref{eq:bc}, then by Theorem~\ref{thm:mr} it satisfies
the Euler--Lagrange equation
\begin{equation*}
\dd \dd y(t)=1
\end{equation*}
for all $t\in\left\{q^{[k]_n}: k\in\mathbb{N}_{0}\right\}\cup\{0\}$.
Applying Theorem~\ref{forward sum-con-multi-divi} (see also
Example~\ref{derivative}) and Theorem~\ref{Anti derivative},
we obtain
\begin{equation*}
y(t)=-\sum_{k=0}^{\infty}\ffkt\left(q^{n^k} t^{n^k (n-1)}-1\right)
\left(q^{[k]_n}t^{n^k}+c\right),
\end{equation*}
where $c$ satisfies equation
\begin{equation*}
\beta=-\sum_{k=0}^{\infty}\left(q^{[k+1]_n}-q^{[k]_n}\right)
\left(q^{[k]_n}+c\right) ,
\end{equation*}
as the power quantum extremal to problem \eqref{AD}--\eqref{eq:bc}.
For example, choosing $n=1$ and $\beta=\frac{1}{1+q}$
in \eqref{AD}--\eqref{eq:bc}, we get the extremal
\begin{equation*}
y(t)=\frac{t^2}{1+q} \, .
\end{equation*}
\end{exm}

% ----------------------------------------------------------

\subsection{Leitmann's direct optimization method}
\label{Leitmann}

Let $\bar{f}:[a,b]_{\h}\times \mathbb{R} \times \mathbb{R}
\rightarrow \mathbb{R}$. We assume that $(u,v)\rightarrow
\bar{f}(t,u,v)$ is a $C^1(\mathbb{R}^{2}, \mathbb{R})$ function for
any $t \in [a,b]_{\h}$ and $\bar{f}(\cdot,y(\cdot),D_{\h}y(\cdot))$,
$\partial_2\bar{f}(\cdot,y(\cdot),\dd y(\cdot))$, and
$\partial_3\bar{f}(\cdot,y(\cdot),\dd y(\cdot))$ are continuous and
bounded on $[a,b]_{\h}$ for all admissible functions $y(\cdot)$.
Consider the integral
\begin{equation*}
\bar{\mathcal{L}}[\bar{y}]=\int_a^b\bar{f}(t,\bar{y}(\fft),\dd
\bar{y}(t))\ddd t.
\end{equation*}

\begin{lem}[Leitmann's power quantum fundamental lemma]
\label{Fund:lemma:Leit}
Let $y=z(t,\bar{y})$ be a transformation
having an unique inverse $\bar{y}=\bar{z}(t,y)$ for all $t\in
[a,b]_{\h}$ such that there is a one-to-one correspondence
\begin{equation*}
y(t)\Leftrightarrow \bar{y}(t),
\end{equation*}
for all functions $y\in \mathbb{E}$ satisfying \eqref{bc} and all
functions $\bar{y}\in \mathbb{E}([a,b]_{\h},\mathbb{R})$ satisfying
\begin{equation}
\label{bc:trans}
\bar{y}=\bar{z}(a,\alpha), \quad \bar{y}=\bar{z}(b,\beta).
\end{equation}
If the transformation $y=z(t,\bar{y})$ is such that there exists a
function $G:[a,b]_{\h} \times \mathbb{R} \rightarrow \mathbb{R}$
satisfying the functional identity
\begin{equation}
\label{id}
f(t,y(\fft),\dd y(t))-\bar{f}(t,\bar{y}(\fft),\dd \bar{y}(t)) =\dd
G(t,\bar{y}(t)) \, ,
\end{equation}
then if  $\bar{y}^{*}$ yields the extremum of $\bar{\mathcal{L}}$
with $\bar{y}^{*}$ satisfying \eqref{bc:trans},
$y^{*}=z(t,\bar{y}^{*})$ yields the extremum of $\mathcal{L}$ for
$y^{*}$ satisfying \eqref{bc}.
\end{lem}

\begin{proof}
The proof is similar in spirit to Leitmann's proof
\cite{Leitmann67,Leitmann01,MR2035262}. Let  $y\in
\mathbb{E}([a,b]_{\h},\mathbb{R})$ satisfies \eqref{bc} and define
functions $\bar{y}\in \mathbb{E}([a,b]_{\h},\mathbb{R})$ through the
formula $\bar{y}=\bar{z}(t,y)$, $t\in[a,b]_{\h}$. Then $\bar{y}\in
\mathbb{E}([a,b]_{\h},\mathbb{R})$ and satisfies \eqref{bc:trans}.
Moreover, as a result of \eqref{id}, it follows that
\begin{equation*}
\begin{split}
\mathcal{L}[y]-\bar{\mathcal{L}}[\bar{y}] &=
\int_{a}^{b}f(t,y(\fft),\dd y(t))\ddd t
-\int_{a}^{b}\bar{f}(t,\bar{y}(\fft),\dd \bar{y}(t))\ddd t\\
&=\int_{a}^{b}\dd G(t,\bar{y}(t))\ddd t
=G(b,\bar{y}(b))-G(a,\bar{y}(a))\\
&=G(b,\bar{z}(b,\beta))-G(a,\bar{z}(a,\alpha)),
\end{split}
\end{equation*}
from which the desired conclusion follows immediately: the
right-hand side of the above equality is a constant depending only
on the fixed-endpoint conditions \eqref{bc}.
\end{proof}

\begin{exm}
\label{example4} Let $a,b\in I$ with $a<b$,  and $\alpha$, $\beta$
be two given reals, $\alpha \ne \beta$. We consider the following
problem:
\begin{equation}
\label{illust:Ex:4}
\begin{gathered}
\text{minimize} \quad \mathcal{L}[y] =\int_{a}^{b}
\left[\dd (y(t)g(t))\right]^2 \ddd t \, , \\
y(a)=\alpha \, , \quad y(b)=\beta \, ,
\end{gathered}
\end{equation}
where $g$ does not vanish on the interval $[a,b]_{\h}$. Observe that
$\bar{y}(t) = g^{-1}(t)$ minimizes $\mathcal{L}$ with end conditions
$\bar{y}(a) = g^{-1}(a)$ and $\bar{y}(b) = g^{-1}(b)$. Let
$y(t)=\bar{y}(t)+p(t)$. Then
\begin{multline}
\label{tran}
\left[\dd (y(t)g(t))\right]^2 =\left[\dd
(\bar{y}(t)g(t))\right]^2\\
+\dd (p(t)g(t))\dd
\left(2\bar{y}(t)g(t)+p(t)g(t)\right).
\end{multline}
Consequently, if $p(t) = (At + B)g^{-1}(t)$, where $A$ and $B$ are
constants, then \eqref{tran} is of the form \eqref{id}, since $\dd
(p(t)g(t))$ is constant. Thus, the function
\begin{equation*}
y(t)=(At+C)g^{-1}(t)
\end{equation*}
with
\begin{equation*}
A=\left[\alpha g(a)-\beta g(b)\right](a-b)^{-1}, \quad
C=\left[a\beta g(b)-b\alpha g(a)\right](a-b)^{-1},
\end{equation*}
minimizes \eqref{illust:Ex:4}.
\end{exm}

% ----------------------

\section*{Acknowledgements}

The second and third authors were partially supported by the
\emph{Systems and Control Group} of the R\&D Unit CIDMA through the
Portuguese Foundation for Science and Technology (FCT). Malinowska
was also supported by BUT Grant S/WI/2/11; Torres by research
project UTAustin/MAT/0057/2008. We are grateful to an anonymous referee
for several comments and suggestions.

% ----------------------

% ----------------------------------

\end{document}